\documentclass[12pt,twoside]{article}

 \usepackage[all]{xy}
 \usepackage{amsfonts}
\usepackage{amsthm}

\date{March 19, 2014}

\begin{document}

\centerline{\bf International Journal of Algebra, Vol. 8, 2014, no. 4, 195 - 204}

\centerline{\bf HIKARI Ltd, \ www.m-hikari.com}

\centerline{\bf http://dx.doi.org/10.12988/}

\centerline{}

\centerline{}

\centerline {\Large{\bf Induced Representations of Hopf Algebras}}

\centerline{}

\centerline{}

\centerline{\bf {Ibrahim Saleh}}

\centerline{Mathematics Department}

\centerline{University of Wisconsin-Marathon }

\centerline{Wausau, WI 54401}

\newtheorem{thm}{\quad Theorem}[section]

\newtheorem{defn}[thm]{\quad Definition}

\newtheorem{defns}[thm]{\quad Definitions}

\newtheorem{cor}[thm]{\quad Corollary}

\newtheorem{lem}[thm]{\quad Lemma}

\newtheorem{exam}[thm]{\quad Example}

\newtheorem{rem}[thm]{\quad Remark}

 \newtheorem{rems}[thm]{\quad Remarks}

 \newtheorem{exams}[thm]{\quad Examples}

 \newtheorem{conj}[thm]{\quad Conjecture}

 \newtheorem{que}[thm]{\quad Question}
\newcommand{\field}[1]{\mathbb{#1}}

 \newcommand{\eps}{\varepsilon}
 \newcommand{\To}{\longrightarrow}
 \newcommand{\h}{\mathcal{H}}
 \newcommand{\s}{\mathcal{S}}
 \newcommand{\A}{\mathcal{A}}
 \newcommand{\J}{\mathcal{J}}
 \newcommand{\M}{\mathcal{M}}
 \newcommand{\W}{\mathcal{W}}
 \newcommand{\X}{\mathcal{X}}
 \newcommand{\BOP}{\mathbf{B}}
 \newcommand{\BH}{\mathbf{B}(\mathcal{H})}
 \newcommand{\KH}{\mathcal{K}(\mathcal{H})}
 \newcommand{\Real}{\mathbb{R}}
 \newcommand{\Complex}{\mathbb{C}}
 \newcommand{\Field}{\mathbb{F}}
 \newcommand{\RPlus}{\Real^{+}}
 \newcommand{\Polar}{\mathcal{P}_{\s}}
 \newcommand{\Poly}{\mathcal{P}(E)}
 \newcommand{\EssD}{\mathcal{D}}
 \newcommand{\Lom}{\mathcal{L}}
 \newcommand{\States}{\mathcal{T}}
 \newcommand{\abs}[1]{\left\vert#1\right\vert}
 \newcommand{\set}[1]{\left\{#1\right\}}
 \newcommand{\seq}[1]{\left<#1\right>}
 \newcommand{\norm}[1]{\left\Vert#1\right\Vert}
 \newcommand{\essnorm}[1]{\norm{#1}_{\ess}}

\centerline{}

{\footnotesize Copyright $\copyright$ 2014 Ibrahim Saleh. This is an open access article distributed under the Creative Commons Attribution License, which permits unrestricted use, distribution, and reproduction in any medium, provided the original work is properly cited.}

\begin{abstract} Hopf representation is a module and comodule with a consistency condition that is more general  than the consistency condition of  Hopf modules. For a Hopf algebra $H$,  we construct an induced  Hopf representation from a representation of a  bialgebra $B$ using  a bialgebra   epimorphism $\pi: H\rightarrow B$. Application on the quantum group $E_{q}(2)$ is given.
\end{abstract}



\section{Introduction}  Hopf algebras theory provides the ring theoretical ground for quantum groups, which has been one of the most active research topics for the past three decades. The discovery of the quantum groups has enriched the classical representations theory of Lie groups. One of the powerful tools in the theory of unitary representations of classical Lie groups is the induction process.  Induction of quantum groups representations was developed in several articles, for example [3, 6].

 From [5, Section 4.1], for a Hopf algebra $H$, the  module structure on a left $H$-module $W$, induces a specific left module structure on $H\otimes W$, (given in  part (1) of Remark 2.2). $W$ will be a left $H$-\emph{Hopf module} if it caries a comodule structure that is  consistent with the module structure. The consistency condition can be phrased as; the comodule structure defines a module homomorphism from $W$ to $H\otimes W$  with respect to the specific module structure on $H\otimes W$, Definition 2.4.  Theorem 2.3 provides two other equivalent conditions for the consistency condition. In this work, we relax the consistency condition by freeing it from the specific module structure on $H\otimes W$. A \emph{left Hopf representation} $W$ is defined to be a left  $H$-module that has a comudule structure which defines a homomorphism from $W$ to $H\otimes W$ with respect to some module structure on $H\otimes W$, Definition 2.5. Every Hopf module is by definition a Hopf representation but the converse is not necessarily true, Examples 2.6.

  The main result of this article is  constructing an induced Hopf representation for a Hopf algebra $H$ using a representation of a quantum subgroup $(B, \pi)$, where $B$ is a bialgebra and $\pi: H\longrightarrow B$ is a bialgebras epimorphism, Theorem 2.12.

The article is organized as follows. The second section consists of three subsections: first one is devoted for the definition of Hopf modules; in the second subsection we introduce  Hopf representations; and in the last subsection  we provide the induction construction for Hopf representations. We finish the article with an application on the quantum group $E_{q}(2)$. In the following, let $H=(H, \cdot, 1,\triangle ,\epsilon, S )$ be a Hopf algebra over a field $K$ and  $W$ be a $K$-space. We will use Sweedler sigma notations, $id$ as an identity map and $T$ as the twist map.

\section{Hopf  Representations}

\subsection{Hopf Modules}

The material of this subsection can be found in [1, 5].
\begin{defns}

\begin{enumerate}
  \item $(W, \alpha) $ is a \emph{left $H$-module} if the following two diagrams are commutative
  
\begin{equation}
       \xymatrix{
 { K \otimes W} \ar[r]\ar[d]_{1 \otimes id_{W}}& W \\
  {H\otimes W}  \ar[ur]_{\alpha}}; 
\end{equation} and
   \begin{equation}\label{}
   \xymatrix{
 {H\otimes H \otimes W} \ar[r]^{id_{H}\otimes \alpha}\ar[d]_{\cdot  \otimes id_{W}}& {H\otimes W} \ar[d]^{\alpha} \\
  {H\otimes W}  \ar[r]_{\alpha} &  {W}. }
   \end{equation}

  \item $(W, \beta)$ is a \emph{left $H$-comodule} if the following two diagrams are commutative
  
  \begin{equation}\label{}
    \xymatrix{
 {K \otimes W} & W \ar[l] \ar[dl]^{\beta} \\
  {H\otimes W}\ar[u]^{\epsilon \otimes id_{W}} };
  \end{equation}
  and 
  \begin{equation}\label{}
    \xymatrix{
 {H \otimes H \otimes W} & {H\otimes W}\ar[l]_{id_{H} \otimes \beta}  \\
  {H\otimes W} \ar[u]^{\triangle \otimes id_{W}}  &  {W}\ar[l]^{\beta}\ar[u]_{\beta} .}
  \end{equation}

    \item Let $(W_{1}, \alpha_{1})$ and $(W_{2}, \alpha_{2})$ be two left $H$-modules. Then, the map $f: W_{1}\longrightarrow W_{2}$ is called \emph{a left $H$-module homomorphism} with respect to $\alpha_{1}$ and $\alpha_{2}$ if the following diagram is commutative

\begin{equation}\label{}
     \xymatrix{
 {W_{1} } \ar[r]^{f}& {W_{2}}  \\
  {H\otimes W_{1}} \ar[u]^{\alpha_{1}} \ar[r]_{id_{H}\otimes f} &  {H\otimes W_{2}}\ar[u]_{\alpha_{2}}. }
\end{equation}

    \item Let $(W_{1}, \beta_{1})$ and $(W_{2}, \beta_{2})$ be two left $H$-comodules. Then, the map $f: W_{1}\longrightarrow W_{2}$ is called a \emph{left $H$-comodule homomorphism} respect to $\beta_{1}$ and $\beta_{2}$ if the following diagram is commutative

     \begin{equation}\label{}
      \xymatrix{
 {W_{1} } \ar[r]^{f}\ar[d]_{\beta_{1}}& {W_{2}}\ar[d]^{\beta_{2}}  \\
  {H\otimes W_{1}}  \ar[r]_{id_{H}\otimes f} &  {H\otimes W_{2}} .}
     \end{equation}

\end{enumerate}
\end{defns}

\begin{rem}  \begin{enumerate}
  \item
If $(W, \alpha)$ is a left $H$-module, then $(H\otimes W, \phi)$ is a
left $H-$module with

\begin{equation}\label{}
    \phi =(\cdot  \otimes \alpha)\circ (id_{H}\otimes T\otimes id_{W})\circ (\triangle \otimes  id_{H\otimes W}).
\end{equation}
  \item
If $(W, \beta)$ is a left $H$-comodule, then $(H\otimes W, \psi)$ is a
left $H$-comodule with

\begin{equation}\label{}
    \psi =(\cdot \otimes id_{H\otimes W})\circ(id_{H}\otimes T\otimes id_{W})\circ ( \triangle \otimes \beta ).
\end{equation}
\end{enumerate}

\end{rem}

\begin{thm} ([1, 5]) Let  $(W, \alpha)$ be a left $H$-module and $(W, \beta)$ be a left $H$-comodule. Consider $\phi$ and $\psi$ given in (7) and (8) respectively.  Then the following three conditions are equivalent

\begin{enumerate}
  \item
   The following diagram is commutative

   \begin{equation}\label{}
    \xymatrix{
 {H \otimes W }\ar[d]^{\triangle \otimes \beta} \ar[r]^{\alpha}& W \ar[r]^{\beta} & {H\otimes W}  \\
  {H\otimes H \otimes H \otimes W} \ar[r]_{id_{H}\otimes  T \otimes id_{W}}  &  {H\otimes H \otimes H \otimes W}\ar[ur]_{ \cdot \otimes \alpha} ;}
   \end{equation}

  \item $\beta $ is H-module morphism with respect to $\alpha $ and
$\phi $;
  \item $\alpha $ is H-comodule morphism with respect to  $\psi $ and $\beta$.
\end{enumerate}

\end{thm}

\begin{defn} The triple $(W, \alpha, \beta)$ is called a \emph{left $H$-Hopf module}  if one the conditions of Theorem 2.3 is satisfied.
\end{defn}

\subsection{Hopf Representations}

\begin{defn}
\begin{enumerate}
               \item The quadrable $(W, \alpha, \beta, \phi)$ is called a left \emph{Hopf representation of the first type} if the following three conditions are satisfied
               \begin{enumerate}
                 \item $(W, \alpha)$ is a left $H$- module;
                 \item $(W, \beta)$ is a left $H$-comodule;
                 \item $(H\otimes W, \phi)$ is a left $H$-module, such that $\beta $ is H-module morphism with respect to $\alpha $ and $\phi $.

               \end{enumerate}
               \item The quadrable $(W, \alpha, \beta, \psi)$ is called a left \emph{Hopf representation of the second type} if the following three conditions are satisfied
               \begin{enumerate}
                 \item $(W, \alpha)$ is a left $H$- module;
                 \item $(W, \beta)$ is a left $H$-comodule;
                 \item $(H\otimes W, \psi)$ is a left $H$-comodule, such that $\alpha $ is H-comodule morphism with respect to to $\beta$ and $\psi $.

\end{enumerate}
               \item
              The quintuple $(W, \alpha, \beta, \phi, \psi)$ is said to be a left $H$-\emph{Hopf representation} if the following two conditions are satisfied
\begin{enumerate}
  \item $(W,\alpha, \beta, \phi)$ is a left  Hopf $H$-representation of the first type;
  \item $(W,\alpha, \beta, \psi)$ is a left  Hopf $H$-representation of the second type.

\end{enumerate}
 \end{enumerate}
\end{defn}

 Right Hopf representations can be defined in a similar way with the obvious changes.  The second example of the following two examples is  a Hopf representation that is not a Hopf module.

\begin{exams}
 \begin{enumerate}
\item
If $(W, \alpha ,\beta )$ is an $H$-Hopf module, then $(W, \alpha ,\beta, \phi, \psi)$ is an $H$-Hopf representation, where
$\phi $ and $\psi $ are defined as in (7) and (8) respectively.

  \item Let $W=H\otimes H$. Consider the following right $H$-module and $H$-comodule.
\begin{equation}\label{}
\nonumber \alpha: (H\otimes H)\otimes H\longrightarrow H\otimes H
\end{equation}
given by
\begin{equation}\label{}
   \nonumber h\otimes h'\otimes h''\longmapsto h\otimes h'\cdot h'',
\end{equation}

and
\begin{equation}\label{}
 \nonumber  \beta: H\otimes H\longrightarrow H\otimes H\otimes H
\end{equation}

given by
\begin{equation}\label{}
\nonumber h\otimes h'\longmapsto h_{(1)}\otimes h'\otimes h_{(2)}.
\end{equation}
One can see that $(W, \alpha )$ is a right $H$-module. To show that $(W, \beta)$ is a right comodule, first for   diagram  (3) (for right comodules) we have
\begin{eqnarray}
 \nonumber(id_{H\otimes H}\otimes \epsilon)\circ \beta (h\otimes h') &=& (id_{H\otimes H}\otimes \epsilon )(h_{(1)}\otimes h'\otimes h_{(2)}) \\
 \nonumber  &=& h_{(1)}\otimes h'\otimes \epsilon (h_{(2)})\\
  \nonumber &=& h_{(1)} \epsilon (h_{(2)})\otimes h'\otimes 1 \\
 \nonumber &=& h\otimes h'\otimes 1.
\end{eqnarray}

For   diagram (4) (for right comodules) we have
\begin{eqnarray}
 \nonumber(id_{H\otimes H}\otimes \triangle)\circ \beta (h\otimes h') &=& (id_{H\otimes H}\otimes \triangle )(h_{(1)}\otimes h'\otimes h_{(2)})\\
 \nonumber &=& h_{(1)}\otimes h'\otimes h_{(2)(1)}\otimes h_{(2)(2)} \\
  \nonumber &=& h_{(1)(1)}\otimes h'\otimes h_{(1)(2)}\otimes h_{(2)} \\
  \nonumber &=& (\beta \otimes id_{H})\circ \beta (h\otimes h').
\end{eqnarray}
In the following we will show that $(W, \alpha, \beta) $ is not a right Hopf
module by showing that  diagram (9) (for right modules and comodules) is not commutative. Let $x$ be a non unit  element of $H$. Then
\begin{equation}\label{}
    \nonumber (\beta \circ  \alpha )(h\otimes h'\otimes x)=\beta (h\otimes h'\cdot x)=h_{(1)}\otimes h'\cdot x\otimes h_{(2)}.
\end{equation}
However
\begin{eqnarray}
  \nonumber (\alpha \otimes  \cdot )\circ (id_{H\otimes H}\otimes T\otimes id_{H})\circ (\beta \otimes \triangle)
(h\otimes h'\otimes x)\end{eqnarray}
\begin{eqnarray}
 \nonumber &=& (\alpha \otimes  \cdot )\circ (id_{H\otimes H}\otimes T\otimes id_{H})(h_{(1)}\otimes h'\otimes h_{(2)}\otimes x_{(1)}\otimes x_{(2)}) \\
  \nonumber &=& h_{(1)}\otimes h'\cdot x_{(1)}\otimes h_{(2)}\cdot  x_{(2)}.
 \end{eqnarray}
 We define a right Hopf representation structure of the first type on  $(W, \alpha, \beta)$. Consider the map
\begin{equation}\label{}
   \phi: H\otimes H\otimes H\otimes H\longrightarrow H\otimes H\otimes H
\end{equation}

given by
\begin{equation}\label{}
\phi (h\otimes h'\otimes h'' \otimes h''')=h\otimes h'\cdot h'''\otimes h''.
\end{equation}
To see that $(H\otimes H\otimes H, \phi )$ is an $H$-module. One can see that  diagram  (1) (for right modules) is commutative for $\phi$, and for  diagram (2) (for right modules) we have
\begin{eqnarray}
 \nonumber \phi \circ (\phi \otimes id_{H})(h\otimes h'\otimes h''\otimes h'''\otimes \bar{h})  &=&\phi (h\otimes h'\cdot h'''\otimes h''\otimes \bar{h}) \\
   \nonumber &=& h\otimes h'\cdot h'''\cdot \bar {h}\otimes h'' \\
  \nonumber &=& \phi \circ (id_{W\otimes H}\otimes \cdot )(h\otimes h'\otimes  h''\otimes h'''\otimes \bar{h}).
\end{eqnarray}
Finally,  $\beta $ is a right module homomorphism with respect to $\alpha$ and $\phi $, since

\begin{eqnarray}
 \nonumber \beta (\alpha (h\otimes h'\otimes h'')) &=& \beta (h\otimes h'\cdot h'') \\
 \nonumber  &=& h_{(1)}\otimes h'\cdot h''\otimes h_{(2)} \\
 \nonumber &=& \phi (h_{(1)}\otimes h'\otimes h_{(2)}\otimes h'') \\
 \nonumber  &=& \phi (\beta (h\otimes h')\otimes h'').
\end{eqnarray}
Which makes $(W, \alpha, \beta, \phi )$  a first type right Hopf representation.  \  \  \ \ \ \ \        $\Box$

\end{enumerate}

\end{exams}

A subspace $W'$  of a left (respect to right) Hopf representation $W$ is called a sub Hopf representation of $W$ if the restrictions of the module, comodule and consistency structure maps on $W'$ form  Hopf  representation on $W'$. $W$ is called irreducible if it does not have any  nontrivial sub Hopf representation.  The following two remarks are handy.

\begin{rems}\begin{enumerate}
             \item The tensor product of two left (respect to right) Hopf representations is a a left (respect to right) Hopf representation.
             \item If $W'$ is a sub Hopf representation of $W$, then $W/W'$ has a unique Hopf representation structure such that the projection map $\rho: W\rightarrow W/W'$ is a module and comodule map.
           \end{enumerate}
\end{rems}

\subsection{Induced Hopf Representations}

\begin{defn}([3],[6]) Let $(B, \cdot_{B}, 1, \triangle_{B}, \epsilon_{B}) $ be a bialgebra.  The pair  $(B,\pi)$ is called \emph{a quantum subgroup }of $H$ if  $\pi:H\rightarrow B$ is a \emph{bialgebra epimorphism}, that is $\pi$ is an algebra epimorphism such that the following two diagrams are commutative
\begin{equation}\label{}
   \xymatrix{
 {H} \ar[r]^{\pi}\ar[d]_{\epsilon} & B\ar[dl]^{\epsilon_{B}} \\
  {K} };\end{equation}\label{} and   \begin{equation}\label{}  
    \xymatrix{
 { H } \ar[r]^{\pi}\ar[d]_{\triangle}& {B} \ar[d]^{\triangle_{B}} \\
  {W\otimes H}  \ar[r]_{\pi \otimes \pi} &  {B\otimes B}. }
\end{equation}

\end{defn}
In the following, Let  $(B,\pi)$ be a quantum subgroup of $H$ and $(L, \eta)$ be a $B$-module, where $L$ is a Hilbert space.

 The space of the induced representation as a subspace of $H\otimes L$ is given by
 \begin{equation}\label{}
  \nonumber  L'=\{h\otimes l; R(h)\otimes \eta (b\otimes l)=h\otimes 1\otimes \eta (b\otimes l),\quad \forall b\in B\}
 \end{equation}
 where
 \begin{equation}\label{}
    R=(id_{H}\otimes \pi)\circ \triangle .
 \end{equation}
The following two lemmas show that $L'$ is non trivial.
\begin{lem} $1\otimes l\in L'$, \ \  for all $l\in L$.
\end{lem}
\begin{proof} Straightforward from the definition of $L'$.
\end{proof}

\begin{lem} If $h\otimes l \in L'$ then $h\otimes \eta (b\otimes l)\in L',$ for all \ \  $b\in B$
\end{lem}
\begin{proof} Since $h\otimes l\in L'$, then

\begin{equation}\label{}
 \nonumber     R(h)\otimes \eta (b\otimes l)=h\otimes 1\otimes \eta (b\otimes l).
\end{equation}

Now, if $b'$ is an  element of $B$, then
\begin{equation}\label{}
 \nonumber  R(h)\otimes b'\otimes \eta (b\otimes l)=h\otimes 1\otimes  b' \otimes \eta (b\otimes l)).
\end{equation}
Apply the  following map on both sides
\begin{equation}\label{}
 \nonumber   id_{H}\otimes id_{H}\otimes \eta ,
\end{equation}
we get
\begin{equation}\label{}
  \nonumber R(h) \otimes \eta(b' \otimes \eta (b\otimes l))=h\otimes 1\otimes \eta( b' \otimes \eta (b\otimes l)).
\end{equation}
\end{proof}

\begin{lem} If $h\otimes l \in L'$, then $\triangle (h)\otimes l \in H\otimes L'$.
\end{lem}
\begin{proof} Let
\begin{equation}\label{}
 \nonumber f:L\longrightarrow H\otimes L,\qquad f (l)=1\otimes l. \end{equation}

Then,  we have
  \begin{eqnarray}
         \nonumber (id_{H}\otimes id_{H}\otimes f)(\triangle (h)\otimes l) &=& (\triangle \otimes id_{H\otimes L})\circ (id_{H}\otimes f )(h\otimes l)\\
         \nonumber     &=& (\triangle \otimes id_{H\otimes L})\circ (R\otimes id_{L})(h\otimes l)\\
         \nonumber     &=& ((\triangle \otimes id_{H})\circ R )\otimes id_{L})(h\otimes l)\\
         \nonumber     &=& ((id_{H}\otimes R)\circ \triangle )\otimes id_{L})(h\otimes l)\\
        \nonumber      &=& (id_{H}\otimes R\otimes id_{L})(\triangle \otimes id_{L})(h\otimes l).
          \end{eqnarray}
Where the second equation is by using $h\otimes l\in L'$ with  setting $b=1$ in the definition of $L'$, and the fourth equation is by using the  coassociativity  and (20). Now, for $b\in B$, we have

\begin{equation}\label{}
 \nonumber \triangle (h)\otimes 1\otimes b\otimes l=h_{(1)}\otimes R(h_{(2)})\otimes b\otimes l.
\end{equation}
 Hence
\begin{equation}\label{}
  \nonumber  \triangle (h)\otimes 1\otimes \eta (b\otimes l)=h_{(1)}\otimes R(h_{(2)})\otimes  \eta (b\otimes l),
\end{equation}
which finishes the proof.
\end{proof}

\begin{thm} $(L', \alpha, \beta, \phi _{H})$ is a first type Hopf representation of $H$, with
\begin{equation}\label{}
  \nonumber \alpha :H\otimes L'\longrightarrow L', \  where\ \  \alpha (h\otimes h'\otimes l)=h' \otimes \eta (\pi (h)\otimes l),
\end{equation}

\begin{equation}\label{}
\nonumber \beta :L'\longrightarrow H\otimes L',  \  \ where \ \ \beta (h\otimes l)=\triangle (h)\otimes l,
\end{equation}
and the consistency  map
\begin{equation}\label{}
  \nonumber  \phi_{H}:H\otimes H\otimes L'\rightarrow H\otimes L', \  \  by \ \ \phi_{H}(h\otimes h'\otimes h''\otimes l)=h'\otimes h''\otimes \eta(\pi (h)\otimes l).
\end{equation}
\end{thm}
\begin{proof} First, to see that $(L', \alpha) $ is a left $H$-module, we have
\begin{eqnarray}
  \nonumber \alpha \circ (1\otimes id_{L'})(k\otimes h\otimes l)&=& \alpha (k\otimes h\otimes l)\\
   \nonumber &=& h\otimes \eta (\pi (k)\otimes l) \\
   \nonumber &=& k(h\otimes l).
 \end{eqnarray}
  And
  \begin{eqnarray}
   \nonumber  \alpha \circ (id_{H}\otimes \alpha )(h\otimes h'\otimes h''\otimes l) &=& \alpha (h\otimes h''\otimes \eta (\pi (h')\otimes l)) \\
    \nonumber  &=& h''\otimes \eta (\pi (h)\otimes \eta ((\pi(h')\otimes l)) \\
    \nonumber  &=& h'' \otimes \eta \circ (id_{H}\otimes \eta )(\pi (h)\otimes \pi (h')\otimes l) \\
   \nonumber   &=& h''\otimes \eta \circ (\cdot _{B}\otimes id_{L})(\pi (h)\otimes \pi (h')\otimes l) \\
    \nonumber  &=& h''\otimes \eta (\pi (h\cdot h')\otimes l) \\
    \nonumber  &=& \alpha \circ (\cdot \otimes id_{L'})(h\otimes h'\otimes h''\otimes l).
  \end{eqnarray}
 Second, to show that $(L', \beta)$ is a left $H$-comodule, we have
\begin{eqnarray}
(\epsilon \otimes id_{L'})\circ \beta (h\otimes l )&=& (\epsilon \otimes id_{L'})(\triangle(h)\otimes l) \\
 \nonumber &=& \epsilon (h_{(1)})\otimes h_{(2)}\otimes l\\
 \nonumber &=& 1\otimes h\otimes l.
\end{eqnarray}
And
\begin{eqnarray}
 ((\triangle \otimes id_{L'})\circ \beta )(h\otimes l) &=& (\triangle \otimes id_{L'})(h_{(1)}\otimes h_{(2)}\otimes l) \\
 \nonumber&=& \triangle (h_{(1)})\otimes h_{(2)}\otimes l \\
 \nonumber &=& h_{(1)}\otimes \triangle (h_{(2)})\otimes l \\
 \nonumber &=& h_{(1)}\otimes \beta (h_{(2)}\otimes l)\\
 \nonumber &=& (id_{H}\otimes \beta )\circ \beta (h\otimes l).
\end{eqnarray} $(H\otimes L', \phi _{H})$ is  a left $H$-module is straightforward.
Finally, to show that $\beta $ is a module map with respect to  $\alpha $ and $\phi_{H} $, we have
\begin{eqnarray}
\nonumber \beta (\alpha (h \otimes h'\otimes l)) &=& \beta (h'\otimes \eta (\pi (h)\otimes l)) \\
 \nonumber &=& \triangle (h')\otimes \eta (\pi (h)\otimes l) \\
\nonumber &=& \phi _{H} (h\otimes h'_{(1)}\otimes h'_{(2)}\otimes l) \\
\nonumber &=& \phi _{H} (h\otimes \beta (h'\otimes l)) \\
 \nonumber  &=& (\phi _{H}\circ (id_{H}\otimes \beta ))(h\otimes h'\otimes l).
\end{eqnarray}

Hence $(L', \alpha ,\beta ,\phi _{H})$ is a first type left Hopf
representations of $H$.
\end{proof}

\begin{exam} Consider the quantum group $E_{q}(2)$ (over the field of complex numbers $\mathbb{C}$) where the basis of its irreducible representation coincides with the irreducible
representation of the quantum group $A(E_{q}(2))$, where $A(E_{q}(2))$ is the Hopf algebra generated by the elements $z,\bar z,a$ and $\bar a$ with the commutation relations $$z\bar
z=\bar z z=1,a\bar a=\bar aa,za=qaz,az=q\bar za,\bar a\bar z=q\bar z\bar a,z\bar a=q\bar az$$ where $q$ is a real number.\\
The comultiplication is given by
\begin{equation}\label{}
  \nonumber  \triangle (z)=z\otimes z, \triangle (a)=a\otimes 1+z\otimes a.
\end{equation}
 The counit and the antipode are given respectively by
 \begin{equation}\label{}
  \nonumber  \epsilon (z)=1,\epsilon (a)=0,S(z)=\bar z,S(a)=-\bar za.
 \end{equation}
We can choose a quantum subalgebra $A(K)$ corresponding to translations in $A(E_{q}(2))$ which defined as $\mathbb{C}[t,\bar{t}]$. The coproduct $\triangle _{K}$, counit $\epsilon _{K}$ and antipode
$S_{K}$ are given by
\begin{equation}\label{}
 \nonumber    \triangle_{K}(t)=t\otimes 1+1\otimes t, \epsilon _{K}(t)=0 , S _{K}(t)=-t.
\end{equation}
and the projection epimorphism, $\pi: A(E_{q})\longrightarrow \mathbb{C}[t,\bar t]$ is given by $\pi (a)=t,\pi (z)=1$. The Hopf subalgebra $A(K)$ is coabelian and its irreducible corepresentations are
one-dimensional. Now, if $\eta $ is the representation of $\mathbb{C}[t,\bar t]$ in $\mathbb{C}$, (for example $\eta (t\otimes c)=\epsilon (t)c=0$ and $\eta (1\otimes c)=c$) then
we have irreducible first type Hopf representation of $A(E_{K}(2))$ on the sub Hilbert space $L'$.
\end{exam}

{\bf Acknowledgements.} The author was  supported by University of Wisconsin Colleges 2014 Summer Research Grant while preparing this work.

{\bf Received: February 12, 2014}

\end{document}